\documentclass[12pt,reqno,a4paper]{amsart}
\usepackage{amssymb,amsfonts,latexsym,mathrsfs}

\title[Dominance Hierarchy in Coxeter Root Systems]{The Dominance Hierarchy In
Root Systems of Coxeter Groups}

\author{Fu, Xiang}
\dedicatory{\upshape
School of Mathematics and Statistics\\
University of Sydney, NSW 2006, Australia\\[.5em]
\texttt{xifu9119@mail.usyd.edu.au}\\
\texttt{X.Fu@maths.usyd.edu.au}\\[1em]
Preliminary version,
\today
}

\newtheorem{theorem}{Theorem}[section]
\newtheorem{lemma}[theorem]{Lemma}
\newtheorem{proposition}[theorem]{Proposition}
\newtheorem{corollary}[theorem]{Corollary}

\theoremstyle{definition}
\newtheorem{definition}[theorem]{Definition}

\theoremstyle{remark}
\newtheorem{remark}[theorem]{Remark}

\numberwithin{equation}{section}

\newcommand{\Z}{\mathbb{Z}}
\newcommand{\N}{\mathbb{N}}
\newcommand{\R}{\mathbb{R}}
\newcommand{\E}{\mathscr{E}}

\newcommand{\dom}{\unrhd}
\DeclareMathOperator{\PLC}{PLC}
\DeclareMathOperator{\coeff}{coeff}

\DeclareMathOperator{\supp}{supp}
\DeclareMathOperator{\GL}{GL}
\DeclareMathOperator{\spa}{span}

\DeclareMathOperator{\dep}{dp}

\subjclass[2010]{20F55 (20F10, 20F65)}
\keywords{Coxeter groups, root systems, dominance}

\begin{document}

\begin{abstract}
If $x$ and $y$ are roots in the root system with respect to the
  standard (Tits) geometric realization of a Coxeter group $W$, we say
  that $x$ \emph{dominates} $y$ if for all $w\in W$, $wy$ is a
  negative root whenever $wx$ is a negative root. We call a positive root
 \emph{elementary} if it does not dominate any 
  positive root other than itself. The set of all elementary roots is denoted by
$\E$.
  It has been proved by B.~Brink and R.~B.~Howlett (Math.~Ann.~\textbf{296}
(1993), 179--190)
 that $\E$ is finite if (and only if)
  $W$ is a finite-rank Coxeter group. Amongst other things, this
  finiteness property enabled Brink and Howlett to establish the
  automaticity of all finite-rank Coxeter groups. Later Brink has
  also given a complete description of the set $\E$ for arbitrary
  finite-rank Coxeter groups (J.~Algebra \textbf{206} (1998)). However the set
of non-elementary positive
roots has received little attention in the literature. In this paper we answer
a collection of questions concerning the dominance behaviour between such
non-elementary
positive roots. In particular, we show that for any finite-rank Coxeter group
and for any non-negative integer $n$, the set of roots each dominating precisely
$n$ other positive roots is finite. We give upper and lower bounds
for the sizes of all such sets as well as an inductive algorithm for their
computation.
\end{abstract}

\maketitle

\section{Summary of Background Material}
\label{sec:intro}
\begin{definition}\textup{(Krammer \cite{DK94})}
 \label{def:datum}
Suppose that $V$ is a vector space over $\R$ and let $(\,,\,)$ be a bilinear
form on $V$, and let $\Pi$ be a subset of $V$. Then $\Pi$ is called a \emph{root
basis} if the following conditions are satisfied:
\begin{itemize}
 \item [(C1)] $(a, a)=1$ for all $a\in \Pi$, and if $a, b$ are distinct elements
of $\Pi$ then either $(a, b)=-\cos(\pi/m_{ab})$ for some integer
$m_{ab}=m_{ba}\geq 2$, or else $(a, b) \leq -1$ (in which case we define
$m_{ab}=m_{ba}=\infty$);
 \item [(C2)] $0\notin \PLC(\Pi)$, where for any set $A$, $\PLC(A)$ denotes the
set
 $$\{\,\sum\limits_{a\in A} \lambda_a a\mid \text{$\lambda_a \geq 0$ for all
$a\in A$ and $\lambda_{a'}>0$ for some $a'\in A$}\,\}.$$ 
\end{itemize}
\end{definition}

If $\Pi$ is a root basis, then we call the triple $\mathscr{C}=(\,V, \, \Pi,
\,(\,,\,)\,)$ a \emph{Coxeter datum}. Throughout this paper we fix a particular
Coxeter datum $\mathscr{C}$. Observe that (C1) implies that for each $a\in \Pi$,
$a\notin \PLC(\Pi\setminus\{a\})$. Furthermore, (C1) together with (C2) yield
that whenever $a, b\in \Pi$ are distinct then $\{a, b\}$ is linearly
independent. For each $a\in \Pi$ define $\rho_a \in \GL(V)$ by the rule: $\rho_a
x=x-2(x, a)a$, for all $x\in V$. Note that $\rho_a$ is an involution and $\rho_a
a=-a$. The following Proposition summarizes a few useful results:
 
\begin{proposition}\textup{\cite[Lecture 1]{RB96}}
 \label{pp:anu1}
\rm{(i)}\quad Suppose that $a, b\in \Pi$ are distinct such that $m_{ab}\neq
\infty$. Set $\theta =\pi/m_{ab}$. Then for each integer $i$,
$$(\rho_a \rho_b)^i a=\frac{\sin(2i+1)\theta}{\sin \theta}a+\frac{\sin
2i\theta}{\sin\theta}b, $$
and in particular, $\rho_a \rho_b$ has order $m_{ab}$.

\noindent\rm{(ii)}\quad Suppose that $a, b\in \Pi$ are distinct such that
$m_{ab}=\infty$. Set $\theta =\cosh^{-1}(-(a, b))$. Then for each integer $i$,
\begin{equation*}
(\rho_a \rho_b)^i a=
\begin{cases}
 \frac{\sinh(2i+1)\theta}{\sinh \theta}a+\frac{\sinh
2i\theta}{\sinh\theta}b, \text{ if $(a, b) \neq -1$}\\
(2i+1)a+2i b, \text{ }\text{ }\text{ }\text{ }\text{ }\text{ }\text{ }\text{ if $(a, b)=-1$}, 
\end{cases}
\end{equation*}
and in particular, $\rho_a \rho_b$ has infinite order.
\qed
\end{proposition}

Let $G_{\mathscr{C}}$ be the subgroup of $\GL(V)$ generated by the involutions
in the set $\{\,\rho_a\mid a\in \Pi\,\}$.
Let $(W, R)$ be a Coxeter system in the sense of \cite{HH81} or \cite{HM} with
$R=\{\,r_a\mid a\in \Pi\,\}$ being a set of involutions generating $W$
subject only to the condition 
$(r_a r_b)^{m_{ab}}=1$ for all distinct $a, b\in \Pi$ with $m_{ab}\neq \infty$.
Then Proposition~\ref{pp:anu1} yields that there is a group homomorphism
$\phi_{\mathscr{C}}\colon W\to G_{\mathscr{C}}$ satisfying
$\phi_{\mathscr{C}}(r_a)=\rho_a$ for all $a\in \Pi$. This homomorphism together
with the $G_{\mathscr{C}}$-action on $V$ give rise to a $W$-action on $V$: for
each $w\in W$ and $x\in V$, define $wx\in V$ by $wx=\phi_{\mathscr{C}}(w)x$. It
can be easily checked that this $W$-action preserves $(\,,\,)$.
Denote the length function of $W$ with respect to $R$ by $\ell$. Then we have:

\begin{proposition}\textup{\cite[Lecture 1]{RB96}}
 \label{pp:anu2}
Let $G_{\mathscr{C}}, W, R$ be as the above and let $w\in W$ and $a\in \Pi$. If
$\ell(wr_a)\geq \ell(w)$ then $wa\in \PLC(\Pi)$.
\qed
\end{proposition}
\begin{corollary} \textup{(\cite[Lecture 1]{RB96})}
 \label{co:anu2}
$\phi_{\mathscr{C}}\colon W\to G_{\mathscr{C}}$ is an isomorphism.
\end{corollary}
\begin{proof}
 All we need to show is that $\phi_{\mathscr{C}}$ is injective. Let $w\in W$
such that $w a=a$ for all $a\in \Pi$. If $w\neq 1$ then $\ell(w)\geq 1$, and so
we can write $w= w' r_a$ with $a\in \Pi$ and $\ell(w')=\ell(w)-1$. Since
$\ell(w' r_a)>\ell(w')$ the above proposition yields that $w'a \in \PLC(\Pi)$;
but then 
$$ a =w a =w'r_a a =w'(-a)=-w'a,$$
implying $0=a+w'a \in \PLC(\Pi)$, contradicting (C2) of the definition of a root
basis. 
\end{proof}

In particular, the above corollary yields that $(G_{\mathscr{C}}, \{\,\rho_a\mid
a\in \Pi\, \})$ is a Coxeter system isomorphic to $(W, R)$. We call $(W, R)$ the
\emph{abstract Coxeter system} associated to the Coxeter datum $\mathscr{C}$ and
we call $W$ a \emph{Coxeter group} of rank $\#R$, where $\#$ denotes
cardinality. 

\begin{definition}
\label{def:root system}
The \emph{root system} of $W$ in $V$ is the set 
$$\Phi=\{\,wa \mid \text{$w\in W$ and $a\in \Pi$}\,\}.$$
The set $\Phi^+=\Phi\cap \PLC(\Pi)$ is called the set of \emph{positive roots}
and the set $\Phi^-=-\Phi^+$ is called  the set of \emph{negative roots}.
\end{definition}
From Proposition \ref{pp:anu2} and Corollary \ref{co:anu2} we may readily deduce
that:
\begin{proposition}\textup{(\cite[Lecture 3]{RB96})}
\label{pp:anu3}
\rm{(i)}\quad Let $w\in W$ and $a\in \Pi$. Then 
 \begin{equation*}
\ell(wr_a) =
\begin{cases}
\ell(w)-1  \text{   if } wa\in \Phi^-\\
\ell(w)+1  \text{   if } wa\in \Phi^+.
\end{cases}
\end{equation*}

\noindent\rm{(ii)}\quad $\Phi=\Phi^+\biguplus\Phi^-$, where $\biguplus$ denotes
disjoint union.

\noindent\rm{(iii)}\quad $W$ is finite if and only if $\Phi$ is finite.
\qed
\end{proposition}

Let $T=\bigcup_{w\in W} w Rw^{-1}$, and we call it the set of \emph{reflections}
in $W$. For $x\in \Phi$, let $\rho_x\in \GL(V)$ be defined by the rule: $\rho_x
(v)=v-2(v,x)x$, for all $v\in V$. Since $x\in \Phi$, it follows that $x=wa$ for
some $w\in W$ and $a\in \Pi$. Direct calculations yield that $\rho_x
=(\phi_{\mathscr{C}} (w)) \rho_a (\phi_{\mathscr{C}} (w))^{-1}\in G_{\mathscr{C}}$.
Now let $r_x\in W$ such that $\phi_{\mathscr{C}}(r_x)=\rho_x$. Then $r_x = w r_a
w^{-1}\in T$ and we call it the reflection corresponding to $x$. It is readily
checked that $r_x =r_{-x}$ for all $x\in \Phi$ and $T=\{\,r_x \mid x\in
\Phi\,\}$. For each $t\in T$ we let $\alpha_t$ be the unique positive root with
the property that $r_{\alpha_t}=t$. It is also easily checked that there is a
bijection $T \leftrightarrow \Phi^+ $ given by $t \to \alpha_t $ ($t\in T$), and 
$x\to \phi_{\mathscr{C}}^{-1}(\rho_x)$ ($x\in \Phi^+ $). We call this bijection
the \emph{canonical bijection} between $T$ and $\Phi^+$.

For each $x\in \Phi^+$, as in \cite{BH93}, we define the \emph{depth} of $x$
relative to $R$ to
be $\dep(x)=\min\{\, \ell(w) \mid\text{$w\in W$ and $wx\in \Phi^-$}\, \}$. For
$x, y\in \Phi^+$,
we say that $x$ \emph{precedes} $y$, written $x\prec y$ if and only if the
following condition
holds: there exists $w\in W$ such that $y=wx$ and $\dep(y)=\ell(w)+\dep(x)$. It
is readily seen that precedence is a partial order on $\Phi^+$, and the next
result is
taken from \cite{BH93}:
\begin{lemma}\textup{(\cite[Lemma 1.7 ]{BH93}).}
\label{lem:pre}
Let $r\in R$ and $\alpha \in \Phi^+\setminus\{\alpha_r\}$. Then
\begin{equation*}
\dep(r\alpha) =
\begin{cases}
\dep(\alpha)-1  \text{   if } (\alpha, \alpha_r)>0,\\
\dep(\alpha) \ \ \ \ \ \text{         if } (\alpha, \alpha_r)=0,\\
\dep(\alpha)+1 \text{   if } (\alpha, \alpha_r)<0.
\end{cases}
\end{equation*}
\qed
\end{lemma}

Define functions $N\colon W\to \mathcal{P}(\Phi^+)$ and
$\overline{N}\colon W\to \mathcal{P}(T)$ (where $\mathcal{P}$ denotes power set)
by setting
$N(w)=\{\,x\in \Phi^+\mid wx\in \Phi^-\,\}$ and
$\overline{N}(w)=\{\,t\in T \mid \ell(wt)<\ell(w)\,\}$ for all $w\in W$.
Standard arguments as those used in \cite{HM} yield that for each $w\in W$,
$\ell(w)=\#N(w)$ and $\overline{N}(w)=\{\,r_x\mid x\in N(w)\,\}$. In particular,
$N(r_a)=\{a\}$ for each $a\in \Pi$. Furthermore,
$\ell(wv^{-1})+\ell(v)=\ell(w)$, for some $w, v\in W$, if and only if
$N(v)\subseteq N(w)$.
  
A subgroup $W'$ of $W$ is a \emph{reflection subgroup} of $W$ if  $W'=\langle
W'\cap T\rangle$ ($W'$ is generated by the reflections that it contains). For
any reflection subgroup $W'$ of $W$, let 
\begin{align*}
S(W')&=\{\,t\in T\mid \overline{N}(t)\cap W'=\{t\}\,\}\\
\noalign{\hbox{and}}
\Delta(W')&=\{\,x\in \Phi^+\mid r_x\in S(W')\,\}.
\end{align*}
It was shown by Dyer in \cite{MD90} and Deodhar in \cite{Deo} that $(W', S(W'))$
forms a Coxeter system:

\begin{theorem}\textup{(Dyer)}
\label{th:croots}
\rm{(i)}\quad Suppose that $W'$ is a reflection subgroup of $W$. Then 
 $(W', S(W'))$ forms a Coxeter system, and furthermore,
$W'\cap T=\bigcup_{w\in W'}w S(W') w^{-1}$.

\noindent\rm{(ii)} Suppose that $W'$ is a reflection subgroup of $W$ and suppose
that $a, b\in \Delta(W')$ are distinct. Then 
$$
(a, b)\in \{\,-\cos(\pi/n)\mid \text{$n\in \N$ and $n\geq 2$}\,\}\cup (-\infty,
-1].
$$
And conversely if $\Delta$ is a subset of $\Phi^+$ satisfying the condition
that 
$$
(a, b)\in \{\,-\cos(\pi/n)\mid \text{$n\in \N$ and $n\geq 2$}\,\}\cup (-\infty,
-1]
$$
for all $a, b\in \Delta$ with $a\neq b$, then $\Delta =\Delta(W')$ for some
reflection subgroup $W'$ of $W$. In fact $W'=\langle \{\, r_a\mid a\in
\Delta\,\}\rangle$. 
\end{theorem}
\begin{proof}
\rm{(i)}\quad \cite[Theorem 3.3]{MD90}.

\noindent\rm{(ii)}\quad \cite[Theorem 4.4]{MD90}.

\end{proof}

Suppose that $W'$ is a reflection subgroup of $W$ and suppose that $(\,,\,)'$ is
the restriction of $(\,,\,)$ on the subspace of $V$ spanned by $\Delta(W')$.
Then $\mathscr{C'}=(\,\spa(\Delta(W')),\, \Delta(W'),\, (\,,\,)'\,)$ is a
Coxeter datum with $(W', S(W'))$ being the associated abstract Coxeter system.
Consequently the notion of a root system applies to  $\mathscr{C'}$. We let
$\Phi(W')$, $\Phi^+(W')$ and $\Phi^-(W')$ be, respectively, the set of roots,
positive roots and negative roots for the datum $\mathscr{C'}$. Then it follows
from Definition~\ref{def:root system} that $\Phi(W')=W'\Delta(W')$,
$\Phi^+(W')=\Phi(W')\cap \PLC(\Delta(W'))$ and $\Phi^-(W')=-\Phi^+(W')$. Note
that Theorem~\ref{th:croots}~(i) yields that 
$$\Phi(W')=\{\,x\in \Phi\mid r_x \in W'\,\}.$$
We call $S(W')$ the set of \emph{canonical generators} of $W'$, and we call
$\Delta(W')$ the set of \emph{canonical roots} of $\Phi(W')$ (note that
$\Delta(W')$ forms a root basis for the Coxeter datum $\mathscr{C}'$). In this
paper a reflection subgroup $W'$ is called a \emph{dihedral reflection subgroup}
if $\#S(W')=2$. 

A subset $\Phi'$ of $\Phi$ is called a \emph{root subsystem} if $r_y x\in \Phi'$
whenever $x, y$ are both in $\Phi'$. It is easily seen that there is a bijective
correspondence between reflection subgroups $W'$ of $W$ and root subsystems
$\Phi'$ of $\Phi$ given by $W'\mapsto \Phi(W')$ and $\Phi'\mapsto \langle\{\,
r_x\mid x\in \Phi'\,\}\rangle$. 

Theorem \ref{th:croots} (ii) yields that if $a, b\in \Phi^+$ then $\{a, b\}$
forms the set of canonical roots for the dihedral reflection subgroup
$\langle\{r_a, r_b\}\rangle$ generated by $r_a$ and $r_b$ if and only if 
$(a, b) =-\cos(\pi/n)$ for some integer $n\geq 2$ or else $(a, b)\leq -1$.
Observe that in either of these cases, $\{a, b\}$ is linearly independent. In
the former case a similar calculation as in Proposition~\ref{pp:anu1}~(i) yields
that $(r_a r_b)^n$ acts trivially on $V$, furthermore, the dihedral reflection
subgroup $\langle \{r_a, r_b\}\rangle$ is finite. In the latter case, let
$\theta=\cosh^{-1}(-(a,b))$, and for each integer $i$, we employ the following 
notation throughout this paper:
\begin{equation}
\label{eq:c}
 c_i=
\begin{cases}
 \frac{\sinh (i\theta)}{\sinh\theta}, \text{ if $\theta \neq 0$};\\
 i,\quad\quad\text{ }\text{ if $\theta =0$}.
 \end{cases}
\end{equation}
Then a similar calculation as in
Proposition~\ref{pp:anu1}~(ii) yields that for each $i$, 
\begin{equation}
\label{eq:rank2}
 \left\{
\begin{array}{rl}
(r_a r_b)^i a &= c_{2i+1}a + c_{2i}b;\\
r_b(r_a r_b)^i a &= c_{2i+1}a + c_{2i+2}b;\\
(r_b r_a)^i b &= c_{2i}a +c_{2i+1}b;\\
r_a(r_b r_a)^i b &= c_{2i+2}a +c_{2i+1}b.
\end{array} \right.
\end{equation}
It is well known (and can be easily deduced from (\ref{eq:rank2})) that 
\begin{equation}
\label{eq:rootsub}
\Phi(\langle \{r_a, r_b\}\rangle) = \{\, c_i a + c_{i\pm 1} b \mid i \in \Z\,\}.
\end{equation}
Since $c_i>0$ for all $i>0$, it follows from (\ref{eq:rank2}) and the fact that
$\{a, b\}$ is linearly independent that $r_ar_b$ has infinite order, and
consequently $\langle\{r_a, r_b\}\rangle$ is an infinite dihedral reflection
subgroup of $W$. Observe that $c_i\neq c_j$ whenever $i\neq j$, hence
(\ref{eq:rank2}) yields that $a$ and $b$ are not conjugate to each other under
the action of $\langle \{r_a, r_b\}\rangle$, and consequently $\langle \{r_a,
r_b\}\rangle$ has two orbits on $\Phi(\langle\{r_a, r_b\}\rangle)$, one
containing $a$ and the other containing $b$. The root $c_i a+c_{i\pm 1}b$ lies
in the former orbit if and only if $i$ is odd, and it lies in the latter orbit
if and only if $i$ is even. 

For the rest of this section we assume that $a, b\in \Phi^+$ with $(a, b)\leq
-1$ and we keep all the notation of the preceding paragraph. 
\begin{proposition}
\label{pp:rank2}
Suppose that $W'$ is a reflection subgroup of the dihedral reflection subgroup 
$\langle \{r_a, r_b\} \rangle$. Then $\#S(W')\leq 2$.
\end{proposition}
\begin{proof}
Suppose for a contradiction that there are at least three canonical generators 
$x$, $y$ and $z$ for the subsystem $\Phi'$. Then from (\ref{eq:rootsub}) we know
that  there are three integers $m$, $n$ and $p$ with $x= c_m a + c_{m\pm 1} b$,
$y= c_n a + c_{n\pm 1} b$ and $z= c_p a + c_{p \pm 1}b$. If either
\begin{equation*}
 \left\{
\begin{array}{rl}
x&=c_m a + c_{m+1}b\\
y&=c_{n}a + c_{n+1}b
\end{array} \right.\quad\text{or}\quad
\left\{
\begin{array}{rl}
x&=c_m a + c_{m-1}b\\
y&=c_{n}a + c_{n-1}b
\end{array} \right.,
\end{equation*}
then either $(x,y)=\cosh((m-n)\theta)\geq 1$ (if $\theta \neq 0$), or else $(x, y)=1$ 
(if $\theta =0$), resulting in a contradiction to Theorem
\ref{th:croots} (ii). Without loss of generality, we may assume that $x=c_m a+
c_{m+1}b$ and $y=c_n a +c_{n-1}b$. Now if $z=c_p a+c_{p+1}b$, then a short
calculation yields that either $(x,z) =\cosh((m-p)\theta)\geq 1$ (if $\theta \neq 0$), or else
$(x, z)=1$ (if $\theta =0$), a contradiction to
Theorem~\ref{th:croots}~(ii); on the other hand if $z=c_p a+ c_{p-1}b$ then, as before, either $(z,
y) =\cosh((n-p)\theta)\geq 1$ (if $\theta \neq 0$), or else $(z, y)=1$ (if $\theta =0$), 
again a contradiction to
Theorem~\ref{th:croots}~(ii). 
\end{proof}

We close this section with an explicit calculation of the canonical roots for an
arbitrary dihedral reflection subgroup of $\langle\{r_a, r_b\}\rangle$. These
technical results will be used in Section~\ref{sec:dom}. Let $\theta=\cosh^{-1}(-(a, b))$, 
as before.

Suppose that $x= c_m a+ c_{m+1}b$ and $y=c_n a + c_{n-1}b$ are positive roots in
$\Phi(\langle \{r_a, r_b \}\rangle)$ (that is, $m$ is a non-negative integer and
$n$ is a positive integer). Then either $(x,y)=-\cosh((m+n)\theta)\leq -1$ 
(when $\theta \neq 0$), or else $(x, y) = -1$ (when $\theta =0$), and it
follows from Theorem~\ref{th:croots}~(ii) that $\{x, y\} =\Delta (\langle \{
r_x, r_y\}\rangle)$.

Suppose that $x = c_m a + c_{m+1}b$ and $y = c_n a + c_{n+1}b$ are roots in
$\Phi(\langle \{r_a, r_b\}\rangle)$ (with $n<m \in \Z$). Put $d=m-n$.
Proposition~\ref{pp:anu1}~(ii) yields that
\begin{equation}
 \label{eq:rt1}
\Phi(\langle \{r_x, r_y\} \rangle)=\{\, c_{kd-m}a +c_{kd-m-1}b,
c_{kd+m}a+c_{kd+m+1}b \mid k \in \Z \,\}.
\end{equation}
Let $\alpha$, $\beta$ be the canonical roots for this root subsystem. Then we
claim that $\alpha = c_i a+c_{i-1}b$ and $\beta=c_j a+c_{j+1}b$ for some
positive integer $i$ and nonnegative integer $j$. Indeed, (\ref{eq:rootsub})
yields that the only other possibilities are either
\begin{equation*}
 \left\{
\begin{array}{rl}
\alpha&=c_i a + c_{i+1}b\\
\beta&=c_{j}a + c_{j+1}b
\end{array} \right.\quad\text{or}\quad
\left\{
\begin{array}{rl}
\alpha&=c_i a + c_{i-1}b\\
\beta&=c_{j}a + c_{j-1}b
\end{array} \right.,
\end{equation*}
and in either of these two cases, either $(\alpha, \beta)=\cosh((i-j)\theta)\geq 1$,
 or else $(\alpha, \beta)=1$, both contradicting Theorem~\ref{th:croots}~(ii). 
Therefore our claim holds, and in
view of (\ref{eq:rt1}) we have
\begin{equation}
\label{eq:canrt1}
\left\{
\begin{array}{rl}
\alpha&=c_{k_1(m-n)-m}a+c_{k_1(m-n)-m-1}b\\
\beta &=c_{k_2(m-n)+m}a +c_{k_2(m-n)+m+1}b
\end{array} \right.,
\end{equation}
for some integers $k_1$ and $k_2$. In fact, $k_1$ and $k_2$ satisfy the
condition that $k_1(m-n)-m$ is the smallest positive integer of this form and
$k_2(m-n)+m$ is the smallest non-negative integer of this form.

Suppose that $x=c_{m+1}a+c_m b$ and $y=c_{n+1}a+c_n b$ are roots in
$\Phi(\langle \{r_a, r_b\}\rangle)$ (with $n,m\in \Z$). Put $d=m-n$.
Interchanging the roles of $a$ and $b$ in the preceding paragraph, we see that
\begin{equation}
 \label{eq:rt2}
\Phi(\langle \{r_x, r_y\} \rangle)=\{\, c_{ld+m+1}a +c_{ld+m}b,
c_{ld-m-1}a+c_{ld-m}b \mid k \in \Z \,\}.
\end{equation}

Let $\alpha'$, $\beta'$ be the canonical roots for this root subsystem. Exactly
the same reasoning as in the preceding paragraph yields that 
\begin{equation}
\label{eq:canrt2}
\left\{
\begin{array}{rl}
\alpha'&=c_{l_1(m-n)+m+1}a+c_{l_1(m-n)+m}b\\
\beta'&=c_{l_2(m-n)-m-1}a+c_{l_2(m-n)-m}b
\end{array} \right.,
\end{equation}
for some integers $l_1$ and $l_2$. Indeed $l_1$ and $l_2$ satisfy the conditions
that $l_1(m-n)+m$ is the smallest non-negative integer of this form and
$l_2(m-n)-m$ is the smallest positive integer of this form.

\section{Canonical Coefficients}

For a Coxeter datum $\mathscr{C}=(\, V,\, \Pi, \, (\,,\,)\,)$, since $\Pi$ may
be linearly dependent, the expression of a root in $\Phi$ as a linear
combination of elements of $\Pi$  may not be unique. Thus the concept of the
coefficient of an element of $\Pi$ in any given root in $\Phi$  is potentially
ambiguous. This section gives a canonical way of expressing a root in $\Phi$  as
a linear combination of elements from $\Pi$. This canonical expression follows 
from a standard construction similar to the one considered in \cite{HT97}. 

Given a Coxeter datum $\mathscr{C}=(\, V,\, \Pi, \, (\,,\,)\,)$, let $E$ be a
vector space over $\R$ with basis $\Pi_E=\{\,e_a\mid a\in \Pi\,\}$ in bijective
correspondence with $\Pi$ and let $(\,,\,)_E$ be the unique bilinear form on $E$
satisfying 
$$
(e_a, e_b)_E =(a, b) \text{, for all } a, b\in \Pi.
$$
Then $\mathscr{C}_E=(\, E,\, \Pi_E, \, (\,,\,)_E\,)$ is a Coxeter datum.
Moreover, $\mathscr{C}_E$ and $\mathscr{C}$ are
associated to the same abstract Coxeter system $(W, R)$. Corollary~\ref{co:anu2}
yields that 
$\phi_{\mathscr{C}_E}\colon W\to G_{\mathscr{C}_E}=
\langle\{\, \rho_{e_a}\mid a\in \Pi\,\}\rangle$ is an isomorphism.
Furthermore,
$W$ acts faithfully on on $E$ via $r_a y =\rho_{e_a} y$ for all $a\in \Pi$ and
$y\in E$.

Let $f\colon E\to V$ be the unique linear map satisfying $f(e_a)= a$, for all
$a\in \Pi$. It is readily checked that $(f(x), f(y))=(x, y)_E$, for all $x, y\in
E$. 
Now  for all $a\in \Pi$ and $y\in E$, 
\begin{align*}
 r_a(f(y))=\rho_a(f(y))=f(y)-2(f(y), a)a &=f(y)-2(f(y), f(e_a))f(e_a)\\
                                         &=f(y-2(y, e_a)_E e_a)\\
                                         &=f(\rho_{e_a}y)\\
                                         &=f(r_a y).
\end{align*}
Then it follows that $wf(y)=f(wy)$, for all $w\in W$ and all $y\in E$, since
$W$ is generated by $\{\,r_a\mid a\in \Pi\,\}$.
Let $\Phi_E$ denote the root system associated to the datum $\mathscr{C}_E$, 
and let $\Phi_E^+$ (respectively, $\Phi_E^-$) denote the corresponding set of 
positive roots (respectively, negative roots). Then a similar 
reasoning as that of Proposition~2.9 of \cite{HT97} enables us to have:

\begin{proposition}
\label{pp:eqv}
The restriction of $f$ defines a $W$-equivariant bijection $\Phi_E\to \Phi$.
\end{proposition}
\begin{proof}
Since $f(w e_a)=w a$ for all $w\in W$ and $a\in \Pi$, it follows that 
$f(\Phi_E)=\Phi$. Proposition~\ref{pp:anu3} applied to $\mathscr{C}_E$ 
yields that, $w e_a\in \Phi_E^+$ if and only if 
$\ell(w r_a)=\ell(w)+1$, and this  happens if and only if
$w a\in \Phi^+$, so $f(\Phi_E^+)=\Phi^+$.
We are done if we can show that the restriction of $f$ on $\Phi_E^+$ is injective.
Suppose that there are $x, y\in \Phi_E^+$ with $f(x)=f(y)$. Then
$
\phi_{\mathscr{C}}\phi_{\mathscr{C}_E}^{-1}(\rho_x)=\rho_{f(x)}=
\rho_{f(y)}=\phi_{\mathscr{C}}\phi_{\mathscr{C}_E}^{-1}(\rho_y)
$.
Since $\phi_{\mathscr{C}}$ is an isomorphism, it follows that 
$\phi_{\mathscr{C}_E}^{-1}(\rho_x)=\phi_{\mathscr{C}_E}^{-1}(\rho_y)$,  
that is, $x$ and $y$ correspond to the same reflection in $W$. 
Since $x, y\in \Phi_E^+$, it follows that $x=y$, as required. 
\end{proof}

Since $\Pi_E$ is linearly independent, it follows that each root $y\in \Phi_E$
can be written uniquely as $\sum_{a\in \Pi} \lambda_a e_a$; we say that
$\lambda_a$ is the \emph{coefficient} of $e_a$ in $y$ and it is denoted by
$\coeff_{e_a}(y)$. We use this fact together with the $W$-equivariant bijection
$f\colon \Phi_E\leftrightarrow \Phi$ to give a canonical expression of a root in
$\Phi$ in terms of $\Pi$:

\begin{definition}
 Suppose that $x\in \Phi$. For each $a\in \Pi$, define the \emph{canonical
coefficient} of $a$ in $x$, written $\coeff_a(x)$ by requiring that
$\coeff_a(x)=\coeff_{e_a}(f^{-1}(x))$. The \emph{support}, written $\supp(x)$ is
the set of $a\in \Pi$ with $\coeff_a(x)\neq 0$.
\end{definition}

\section{The Dominance Hierarchy}
\label{sec:dom}

\begin{definition}
  \label{def:dn}
\rm{(i)}\quad  For $x$ and $y \in \Phi$, we say that $ x$ \emph{dominates} $y$
with respect to $W$ if
 $\{\, w \in W \mid w x \in \Phi^{-} \,\} \subseteq \{\, w \in W \mid w y \in
\Phi^{-} \,\}$. If $x$ dominates $y$ with respect to $W$ then we write $x\dom
y$.

\noindent\rm{(ii)}\quad For each $x\in \Phi^+$, set
 $
D(x)=\{\, y \in \Phi^+ \mid y\neq x \text{ and } x \dom y\,\},
$
and if $x\in \Phi^+$ and $D(x)=\emptyset$ then $x$ is called \emph{elementary}.
For each $n\in \N$, define
$
D_n=\{\,x\in\Phi^+\bigm\vert \#D(x)=n\,\}
$.

\end{definition}

Note that $D_0$ here is the same set as $\E$ of \cite{BH93} and \cite{BB98}.  
In \cite{BH93} and \cite{BB98} dominance is only defined on $\Phi^+$, and it is
found in \cite{BH93} that dominance is a partial order on $\Phi^+$. Here we have
generalized the notion of dominance to the whole of $\Phi$, as was considered in, 
for example, \cite{HT97}. It can be readily
seen that this generalized dominance is a partial order on $\Phi$. 
Observe that it is clear from the above definition that 
\begin{equation*}
\Phi^+ = \biguplus_{n\in \N} D_n.
\end{equation*}
The set $D_0$ has been properly investigated in \cite{BH93} and \cite{BB98}: if
$W$ is finite then $D_0 =\Phi^+$ (that is, if $W$ is finite, then there is no
non-trivial dominance among its roots), whereas if $W$ is an infinite Coxeter
group of finite rank, then $\# D_0 < \infty$ and furthermore, we can explicitly
compute $D_0$. Observe that in the latter case $\biguplus_{n \in \N, n \geq 1}
D_n$ will be an infinite set. One major result of this paper
(Theorem~\ref{th:dnfinite} below)
is that if $R$ is finite then $D_n$ is finite for all natural numbers~$n$. We
also give upper and lower bounds
on~$\#D_n$ (Corollary~\ref{cor:cor7} and Corollary~\ref{cor:c19} below). But
first we need a few elementary results:

\begin{lemma}
  \label{lem:basicdom}
\rm{(i)}\quad If $x$ and $y \in \Phi^+$, then $x \dom y$ if and only if $(x,y)
\geq 1$ and $\dep(x)\geq \dep(y)$ (with equality on depth if and only if $x=y$).

\noindent\rlap{\rm{(ii)}}\qquad Dominance is $W$-invariant: if $x \dom y$ then
    $wx \dom wy$ for any $w \in W$.

\noindent\rlap{\rm{(iii)}}\qquad Suppose that $x, y \in \Phi$, and $x \dom y$.
Then
    $-y \dom -x$.

\noindent\rlap{\rm{(iv)}}\qquad Suppose that $x\in \Phi^+$ and $y\in \Phi^-$.
Then $x\dom y$ if and only if $(x,y)\geq 1$. 

\noindent\rlap{\rm{(v)}}\qquad Let $x,y\in \Phi$. Then there is dominance
between $x$ and $y$ if and only if $(x,y)\geq 1$.

\end{lemma}

\begin{proof}

\noindent\rlap{\rm{(i)}}\qquad Essentially the same reasoning as in \cite[Lemma
2.3]{BH93} applies.
  
\noindent\rlap{\rm{(ii)}}\qquad Clear from the definition of dominance.
  
\noindent\rlap{\rm{(iii)}}\qquad Suppose for a contradiction that there exists
$w \in W$
    such that $w(-y) \in \Phi^-$ and $w(-x) \in \Phi^+$. Then $w(y)
    \in \Phi^+$ yet $w(x) \in \Phi^-$, contradicting the assumption
    that $x \dom y$.

\noindent\rlap{\rm{(iv)}}\qquad Suppose that $x\dom y$. Since dominance is
$W$-invariant, it follows that $r_y x\dom r_y y\in \Phi^+$ and hence $r_y x\in
\Phi^+$. Now part (i) yields that $(r_y x, r_y y)\geq 1$. Since $(\,,\,)$ is
$W$-invariant, it follows that $(x, y)\geq 1$.

Conversely, suppose that $x\in \Phi^+$ and $y\in \Phi^-$ with $(x, y)\geq 1$. Then
clearly $r_y x=x-2(x, y)y \in \Phi^+$. Thus $r_yx$ and $r_y y=-y$ are both
positive. Then it follows from part (i) that there is dominance between $r_yx$
and $r_yy$. Since dominance is $W$-invariant, it follows that there is dominance
between $x$ and $y$. Finally, given that  $x\in \Phi^+$ and $y\in \Phi^-$, it is
clear that $x\dom y$.

\noindent\rlap{\rm{(v)}}\qquad Suppose that $x, y\in \Phi^-$. Then part (i)
yields that there is dominance between $-x$ and $-y$ if and only if $(-x,
-y)=(x,y)\geq 1$. This combined with part~(i) and part~(iv) above yields the
desired result.

\end{proof}

The following is a simple result that we use repeatedly in this paper:
\begin{lemma}
  \label{lem:reflx} Let $x, y \in \Phi$ be distinct with
$x \dom y$ and $y \in D_0$. Then:
  
\noindent\rlap{\rm{(i)}}\qquad  $r_y x \in \Phi^+$;

  \noindent\rlap{\rm{(ii)}}\qquad $(r_y x, x) \leq -1$ and $(r_y x, y) \leq -1$,
and in particular, $r_y x$ cannot dominate either $x$ or $y$. 
\end{lemma}

\begin{proof}
  \noindent\rlap{\rm{(i)}}\qquad Suppose for a contradiction that $r_y x \in
\Phi^-$.  Lemma~\ref{lem:basicdom}~(ii) then yields that $r_y x \dom r_y y =
-y$. Now Lemma \ref{lem:basicdom} (iii) yields that $y \dom -r_yx \in \Phi^+$.
Since $y \in D_0$, this forces $-r_y x = y$, contradicting $x \neq y$.

\noindent\rlap{\rm{(ii)}}\qquad Since $x \dom y$, it follows from Lemma
\ref{lem:basicdom} (v)  that $(x,y)
  \geq 1$. Then $(r_y x, y) = (x,-y) \leq -1$ and hence there is no
  dominance between $r_y x$ and $y$. Also $(r_y x, x) = (x,x) - 2(x,y)^2 \leq
-1$ and thus there is no dominance between $x$ and $r_y x$ either.
\end{proof}

Suppose that $x, y\in \Phi$ with $x\dom y$. It is worthwhile investigating the
connection between this dominance and the canonical generators of the root
subsystem $\Phi(\langle \{r_x, r_y\}\rangle)$.  

\begin{proposition}
  \label{pp:nondom}
  Suppose that $x, y \in \Phi$ are distinct with $x \dom y$. 
Let $a, b$ be the canonical roots for the root subsystem $\Phi(\langle \{r_x,
r_y\}\rangle)$. Then there exists $w \in \langle\{ r_x, r_y \}\rangle$ such that
either 
  \begin{equation*}
 \left\{
\begin{array}{rl}
wx &= a\\
wy &= -b
\end{array} \right.\qquad \text{or else}\qquad
\left\{
\begin{array}{rl}
wx &= b\\
wy &= -a
\end{array} \right..
\end{equation*}
In particular, $(a,b ) = -(x,y)$. 

\end{proposition}

\begin{proof}
By Theorem \ref{th:croots} (ii) we know that 
$$(a, b)\in(-\infty, -1]\cup\{\,-\cos(\pi/ n)\mid \text{$n\in \N$ and
$n\geq2$}\,\}.$$
Suppose for a contradiction that $(a, b)=-\cos(\pi/n)$ for some integer $n\geq
2$. Write $\theta = \pi/n$, and Proposition~\ref{pp:anu1}~(i) yields that 
$$
 \Phi(\langle \{r_a, r_b\}\rangle)=\{\,\frac{\sin(m+1)\theta}{\sin\theta}a+
\frac{\sin m\theta}{\sin\theta}b\mid m\in \N  \text{ and } 0\leq m<2n\,\}.
$$
Hence there are distinct integers $m_1$ and $m_2$ (less than $2n$) such that 

$$
x= \frac{\sin(m_1+1)\theta}{\sin\theta}a+ \frac{\sin
m_1\theta}{\sin\theta}b\quad\text{and}\quad
y=\frac{\sin(m_2+1)\theta}{\sin\theta}a+ \frac{\sin m_2\theta}{\sin\theta}b.
$$
But then $(x, y)=\cos((m_1-m_2)\pi/n)< 1$, contradicting
Lemma~\ref{lem:basicdom}~(v). Thus $(a, b)\leq -1$ and so
Lemma~\ref{lem:basicdom}~(v) yields that $a\dom -b$ and $b\dom -a$. It then
follows readily that there are 
two dominance chains in the root subsystem $\Phi(\langle \{r_a, r_b\}\rangle)$,
namely:
\begin{multline}
\label{eq:seq1}
\cdots \dom r_a r_b r_a (b) \dom r_a r_b (a) \dom r_a (b) \dom a \\
 \dom -b \dom r_b(-a) \dom r_b r_a (-b) \dom  \ \cdots 
\end{multline}
 and 
\begin{multline}
\label{eq:seq2}
\cdots \dom r_b r_a r_b (a) \dom r_b r_a (b) \dom r_b (a) \dom b \\
\dom -a \dom r_a (-b) \dom r_a r_b (-a) \dom \cdots.
\end{multline}
Observe that each element of $\Phi(\langle \{r_a, r_b\}\rangle)$ lies in exactly
one of
the above chains, and the negative of any element of one of these chains lies in
the other. Thus $x',\,y'\in \Phi(\langle \{r_a, r_b\}\rangle)$ are in the same
chain
if and only if $(x',y')\ge 1$ and in different chains if and only if $(x',y')\le
-1$.

From (\ref{eq:seq1}) we see that the roots dominated by $a$ are all negative,
and
from (\ref{eq:seq2}) we see that the roots dominated by $b$ are all negative.
Clearly
we may choose $w\in\langle \{r_a, r_b\}\rangle$ such
that either $wx=a$ or $wx=b$, and since $wx\dom wy$, it follows that either 
\begin{align}
\label{eq:ch1}
 wx &= a \quad\text{and}\quad  wy \in \Phi(\langle \{r_a, r_b\} \rangle)\cap
\Phi^-\\ 
\noalign{\hbox{or}}
\label{eq:ch2}
 wx &= b \quad\text{and}\quad wy \in \Phi(\langle \{r_a, r_b\} \rangle)\cap
\Phi^-.
\end{align}
Suppose that $w x= a$. Then $(a, -wy)=( wx, -wy) =-(x,y) \leq -1$. Now since
$-wy \in \Phi(\langle \{r_x, r_y\} \rangle)\cap \Phi^+$ and $\langle\{ r_a,
r_{wy}\}\rangle=\langle \{ r_x, r_y\}\rangle$, it follows from Theorem
~\ref{th:croots}~(ii) that $\{ a, -wy\} $ is the set of canonical roots for
$\Phi(\langle \{r_x, r_y\} \rangle)$, which then forces that $-wy =b$. 
Similarly, in the case $wx=b$, we may conclude that $wy =-a$.

\end{proof}

\begin{lemma}
\label{lem:consec}
Suppose that $x$, $y \in \Phi$ are distinct with $x\dom y$. Let $a$ and $b$ be
the canonical roots for $\Phi(\langle \{r_x, r_y\} \rangle)$. Then either 
\begin{equation*}
 \left\{
\begin{array}{rl}
x &= c_m a+c_{m+1}b \\
y &= c_{m-1}a+c_m b
\end{array} \right.\quad \text{or}\quad
\left\{
\begin{array}{rl}
x &= c_m a + c_{m-1}b\\
y &= c_{m-1}a +c_{m-2}b
\end{array} \right.,
\end{equation*}
for some integer $m$, where $c_i$ is as
defined in  (\ref{eq:c}) for each integer $i$.
\end{lemma}
\begin{proof}
Proposition \ref{pp:nondom} yields that $(a, b)\leq -1$. 
Since $a, b$ are the canonical roots of $\Phi(\langle \{r_x, r_y\}\rangle)$, it
follows from equation~(\ref{eq:rootsub}) that  
$x=c_ma +c_{m\pm 1}b$ and
$y=c_n a+c_{n \pm 1} b$,
for some integers $m$ and $n$. Let $\theta =\cosh^{-1}(-(a, b))$. If either 
\begin{equation*}
 \left\{
\begin{array}{rl}
x &= c_ma + c_{m+1}b\\
y &= c_na + c_{n-1}b
\end{array} \right.\quad \text{or}\quad
\left\{
\begin{array}{rl}
x &= c_ma + c_{m-1}b\\
y &= c_na + c_{n+1}b
\end{array} \right.,
\end{equation*}
then either $(x,y)=-\cosh((n+m)\theta)\leq -1$ (when $\theta \neq 0$), 
or else $(x, y)=-1$ (when $\theta =0$), contradicting $x\dom y$. Therefore
there are only two possibilities, namely:
\begin{equation}
\label{eq:3.1}
 \left\{
\begin{array}{rl}
x &= c_ma + c_{m+1}b\\
y &= c_na + c_{n+1}b
\end{array} \right.
\end{equation}
or 
\begin{equation}
\label{eq:3.2}
 \left\{
\begin{array}{rl}
x &= c_{m}a + c_{m-1}b\\
y &= c_{n}a + c_{n-1}b.
\end{array} \right.
\end{equation}

First suppose that (\ref{eq:3.1}) is the case. Since $a$ and $b$ are the
canonical roots for $\Phi(\langle \{r_a, r_b\} \rangle )=\Phi(\langle \{r_x,
r_y\} \rangle)$, it follows from equation~(\ref{eq:canrt1}) that there are
integers $k_1$ and $k_2$ such that 
$$1=k_1(m-n)-m \qquad\text{and}\qquad 0=k_2(m-n)+m.$$
But then $k_1+ k_2 = \frac{1}{m-n} \in \Z$. Clearly this is only possible when
$m-n =\pm 1$. On the other hand, since $x \dom y$, it is readily seen that
$m>n$, giving us 
$ x= c_m a+c_{m+1}b$ and $y=c_{m-1}a+c_m b$.
On the other hand, if (\ref{eq:3.2}) is the case, then by taking
equation~(\ref{eq:canrt2}) into consideration, a similar reasoning as above
yields that 
$x=c_m a + c_{m-1}b$ and $y=c_{m-1}a +c_{m-2}b$.
\end{proof}

\begin{remark}
 Let $x$ and $y$ be as in Proposition~\ref{pp:nondom} and Lemma \ref{lem:consec}
above. Then in fact $x$ and $y$ are consecutive terms in precisely one of the
dominance chains (\ref{eq:seq1}) or (\ref{eq:seq2}).
\end{remark}

Now we are ready for the first key result of this paper:
\begin{theorem}
  \label{th:d1}
  $D_1 \subseteq \{\,r_a b \mid a,b \in D_0\,\}$.
  Furthermore, if $\#R<\infty$ then $\#D_1 \leq (\#D_0)^2-\#D_0$.
\end{theorem}

\begin{proof}
  Suppose that $x \in D_1$ and let $D(x)=\{y\}$. Clearly $y \in D_0$. By
  Lemma~\ref{lem:reflx}~(i), we know that $r_y x \in \Phi^+$. Thus to
  prove Theorem~\ref{th:d1}, it suffices to show that $r_y x \in D_0$.

 Suppose for a contradiction that $r_y x \in \Phi^+ \setminus D_0$. Then
 there exists $z \in \Phi^+\setminus\{r_y x\}$ with $r_y x \dom z$.
 Since dominance is $W$-invariant, it follows that $x \dom r_y z$.
 If $r_y z=y$ then $z\in \Phi^-$, contradicting our choice for $z$. Then the
fact $D(x)=\{y\}$ implies that $r_y z\in \Phi^-$ and in particular, $(z, y)>0$.
Since $r_y x\dom z$ and $x\dom y$, it follows from Lemma ~\ref{lem:basicdom}~(i)
that $(r_y x, z)\geq 1$ and $(x, y)\geq 1$. Then
  \begin{displaymath}
    \begin{split}
      1 \leq (r_y x, z)&= (x-2(x,y)y, z)\\
      &= (x,z) - 2(x,y)(y,z),
    \end{split}
  \end{displaymath}
  implying that $1 \leq (x,z)$. Hence Lemma~\ref{lem:basicdom}~(v) yields that
either $x \dom z$ or else $z \dom x$. In the latter case $r_y x \dom z \dom x$,
contradicting Lemma~\ref{lem:reflx}~(ii). On the other hand, if $x\dom z$, then
our construction forces $z=y$. But then $r_y x\dom y$, again contradicting
Lemma~\ref{lem:reflx}~(ii). Thus $r_y x \in D_0$, as required. Since $x \in D_1$
was arbitrary, it follows that $D_1 \subseteq \{\,r_a b \mid a,b \in D_0\,\}$.

  Finally, since $D_1$ does not contain elements of the form $r_a a$,
  $a \in D_0$, it follows that
  \begin{equation}
    \label{eq:d1}
    D_1 \subseteq \{\,r_a b \mid a,b \in D_0\,\} \setminus {-D_0}.
  \end{equation}
  In the case that $\#R < \infty$, Theorem 2.8 of \cite{BH93} yields that $\#D_0
< \infty$, and so it follows from (\ref{eq:d1}) that $\#D_1 \leq
(\#D_0)^2-\#D_0$.
\end{proof}

The above treatment of $D_1$ can be generalized to $D_n$ for arbitrary
$n \in \N$. Indeed we have:

\begin{theorem}
  \label{th:dnfinite}
  For $n \in \N$, 
$$D_n \subseteq \{\,r_a b \mid a\in D_0, b\in \!\!\biguplus\limits_{m\leq n-1}
D_m\,\}.$$ 
\end{theorem}
\begin{proof}
The case $n=1$ has been covered by Theorem \ref{th:d1}, so we may assume that
$n>1$. 

Let $x\in D_n$, and suppose that $D(x)=\{\,y_1, y_2, \ldots, y_n\,\}$, with
$y_n$ being minimal with respect to dominance. Clearly $y_n \in D_0$ and so
Lemma~\ref{lem:reflx}~(i) yields that $r_{y_n} x \in \Phi^+$. Hence either
$r_{y_n} x \in D_0$ or else $r_{y_n} x \in \Phi^+ \setminus D_0$.

If $r_{y_n}x\in D_0$, then 
$$x\in\{\, r_a b \mid \text{$a,b\in  D_0$}\,\}\subseteq \{\,r_a b \mid a\in D_0,
b\in \biguplus_{m\leq n-1} D_m\,\},$$ 
and the desired result clearly follows, given the arbitrary choice of $x$.

If $r_{y_n} x \in \Phi^+ \setminus D_0$, let $z\in D(r_{y_n}x)$. We claim that
there are at most $(n-1)$ possible values for $z$. Observe that this claim
implies the following:  
$$
r_{y_n} x \in \biguplus_{m \leq n-1} D_m,
$$
then it follows immediately that $D_n\subseteq \{\,r_a b \mid a\in D_0, b\in
\biguplus_{m\leq n-1} D_m\,\}$, since $x\in D_n$ was arbitrary.

Thus all it remains to do is to prove the above claim.
Since $r_{y_n}x \, \dom \, z$, Lemma \ref{lem:basicdom} (ii) yields that $x\dom
r_{y_n} z$. Thus either $r_{y_n}z \in\Phi^+$ and in which case $ r_{y_n} z =
y_i$, for $ 1\leq i \leq n-1$; or else $r_{y_n} z \in \Phi^-.$  If $r_{y_n}z \in
\Phi^-$ then clearly $(y_n, z)>0$. Since $r_{y_n}x\dom z$ and $x\dom y_{n}$,
Lemma~\ref{lem:basicdom}~(v) yields that $(r_{y_n}x, z)\geq 1$ and $(x,
y_n)\geq1$. Then
  \begin{displaymath}
    \begin{split}
      1 \leq (r_{y_n} x, z)&= (x-2(x,y_n)y_n, z)\\
      &= (x,z) - 2(x,y_n)(y_n,z),
    \end{split}
  \end{displaymath}
and hence it follows that $(x,z)\geq 1$. Similar to the proof of
Theorem~\ref{th:d1}, we can conclude that  $x \dom z$ and so $z \in \{ y_1,
\cdots,  y_n\}$. Since $x\dom z$ as well as $r_{y_n}x\dom z$, 
Lemma~\ref{lem:reflx}~(ii) yields that $z \in \{y_1, \cdots, y_{n-1}\}$. Summing
up, if $z\in D(r_{y_n} x)$, then
\begin{displaymath}
  \begin{split}
    z\in \, &\{\, r_{y_n} (y_i) \mid r_{y_n}(y_i) \in \Phi^+,\, i\in \{\, 1,
\ldots, n-1\,\}\,\} \\
    \cup \, &\{\,y_i \mid r_{y_n}(y_i) \in \Phi^-,\, i\in \{\,1, \ldots, n-1
\,\}\,\},
  \end{split}
\end{displaymath}
and this is clearly a disjoint union of size $n-1$. Thus $r_{y_n} x \in D_m$,
for some $m\leq n-1$ and the claim is proved.

\end{proof}

Note that for each positive integer $n$, Theorem \ref{th:dnfinite} immediately
yields the following upper bound for the size of the corresponding $D_n$.
\begin{corollary}
 \label{cor:cor7}
Suppose that $\#R <\infty$. Then $\#D_n<\infty$ for all $n\in \N$. Indeed 
$$\#D_n\leq (\#D_0)^{n+1}-(\#D_0)^n.$$
\end{corollary}

\begin{proof}
Clearly $D_i\cap D_j=\emptyset$ whenever $i\neq j$, so Theorem~\ref{th:dnfinite}
yields that $D_n\subseteq \{r_a b \mid a\in D_0, b\in \biguplus_{m\leq
n-1}D_m\}\setminus(\biguplus_{m<n}D_m)$ and the desired result then follows from
a simple induction on $n$.
\end{proof}

Having shown that for all $n\in \N$, $\#D_n<\infty$ if $\#R < \infty$, it is not
immediately clear, at this stage, that for each $n \in \N$, the corresponding
$D_n \neq \emptyset$. Lemma \ref{lem:lemma9} to Corollary \ref{cor:c19} below
will, amongst other things, establish that $D_n \neq \emptyset$ for each $n\in
\N$ if $W$ is an infinite Coxeter group of finite rank.

\begin{lemma}
 \label{lem:lemma9}
For $n \in \N$, $$\{\,wa \mid \text{$a\in D_0$, $w\in W$, $\ell(w)<n$}\,\} \cap
D_n = \emptyset.$$
\end{lemma}

\begin{proof}
Suppose for a contradiction that there exist some $n\in \N$ and $x= wa \in D_n$
such that $a\in D_0$ and $w\in W$ with $\ell(w)<n$.
Suppose that $D(x)=\{\, y_1, \ldots, y_n \,\}$. Since dominance is
$W$-invariant, it follows that $a=w^{-1}x$ dominates all of $w^{-1}y_1,
w^{-1}y_2, \ldots, w^{-1}y_n$.
Note that
$a \notin\{\, w^{-1} y_1, \cdots, w^{-1} y_n\,\}$. Since $a$ is elementary, it
follows that $w^{-1} y_1, \cdots, w^{-1} y_n \in \Phi^-$, that is, $y_1, \cdots,
y_n \in N(w^{-1})$, but this contradicts the fact that
$\#N(w^{-1})=\ell(w^{-1})=\ell(w)< n$.
\end{proof}

\begin{lemma}
\label{lem:lemma10}
$$R D_0 \subseteq -D_0 \uplus D_0 \uplus D_1.$$ 
\end{lemma}
\begin{proof}
Suppose that $r\in R$ and $x\in D_0$ are arbitrary. 
If $r x\in \Phi^+$, then Lemma~\ref{lem:lemma9} above yields that $rx \in
D_0\uplus D_1$.
On the other hand, if $r x \in \Phi^-$, then $x\in \Pi$, which in turn implies
that $r=r_x$ and $r x=-x \in -\Pi\subseteq -D_0$.
\end{proof}

Generalizing Lemma \ref{lem:lemma10}, we have:

\begin{lemma}
\label{lem:lemma11}
For all $n\geq 1$,
$$R D_n \subseteq D_{n-1} \uplus D_n \uplus D_{n+1}. $$ 
\end{lemma}

\begin{proof}
Suppose that $n \geq 1$, and let $x\in D_n$, and $z\in \Pi$ be arbitrary. Since
$x\neq z$, it follows that $r_z x \in \Phi^+$.

Suppose for a contradiction that $r_z x \in D_m$ for some $m\geq n+2$. Let
$D(r_z x)=\{\, y_1, \ldots, y_m \,\}$. Then $x \dom r_z y_1,   \cdots,  r_z
y_m$. Since $x\in D_n$, and $m\geq n+2$, it follows that there are $1\leq i < j
\leq m$ with $r_z y_i \in \Phi^-$ and $r_z y_j \in \Phi^-$. But this is
impossible, since $r_z$ could only make one positive root negative. Therefore we
may conclude that $r_z x \notin D_m$ where $m \geq n+2$. A similar argument also
shows that $r_z x \notin D_{m'}$ where $m' \leq n-2$, and we are done. 
\end{proof}

\begin{lemma}
\label{lem:lemma13}
Suppose that $x$, $y$ are in $\Phi^+$ with $y \preceq x$. Let $w \in W$ be such
that $x = wy$ and $dp(x)=dp(y)+\ell(w)$. Then $y \in D_m$ implies that $x \in
D_n$ for some $n \geq m$. Furthermore, $w D(y) \subseteq D(x)$.
\end{lemma}
\begin{proof}
It is enough to show that the desired result holds in the case that $w=r_a$ for
some $a\in \Pi$. The more general proof then follows from an induction on
$\ell(w)$.

Since $x=r_a y$ and $y\prec x$, Lemma \ref{lem:pre} yields that $(a, y)<0$ and
so Lemma~\ref{lem:basicdom}~(v) yields that $a\notin D(y)$. Let $D(y)=\{\,z_1,
z_2, \ldots, z_m\,\}$. Then the fact $a\in \Pi$ implies $r_a D(y) \subset
\Phi^+$. Since dominance is $W$-invariant, it follows that $x\dom r_a z_i$ for
all $i\in \{\,1, 2\ldots, m\,\}$. Therefore $\{\,r_a z_1, r_a z_2, \ldots, r_a
z_m\,\}\subseteq D(x)$, whence $x\in D_n$ for some integer $n\geq m$, and $r_a
D(y)\subseteq D(x)$. 
\end{proof}

The next proposition, somewhat an analogue to Lemma \ref{lem:pre}, has many
applications, among which, we can deduce, for arbitrary positive root $x$, the
integer $n$ for which $x \in D_n$. Furthermore, it enables us to compute $D(x)$
explicitly  
as well as to obtain an algorithm to compute all the $D_n$'s systematically.

\begin{proposition}
 \label{pp:14}
Suppose that $x \in D_n$ with $n \geq 1$, and $a\in \Pi$. Then 
\begin{itemize}
\item [(i)] $r_a x \in D_{n-1}$ if and only if $(x,a) \geq 1$;
\item [(ii)] $r_a x \in D_{n+1}$ if and only if $(x,a) \leq -1$;
\item [(iii)] $r_a x \in D_n$  \ \  if and only if $(x, a) \in (-1, 1)$.
\end{itemize}
\end{proposition}
\begin{proof}
(i):\qquad Suppose that $x\in D_n$ and $a\in \Pi$ such that $r_a x \in D_{n-1}$.
Let $D(x)=\{\, z_1, z_2, \ldots, z_n\,\}$. Since dominance is $W$-invariant, it
follows that $r_a x \dom r_a z_i$ for all $i\in \{\, 1, 2, \ldots, n\,\}$. Thus
at least one of $r_az_1, \ldots, r_a z_n$ must be negative. Without loss of
generality, we may assume that $r_a z_1 \in \Phi^-$. Since $a\in \Pi$, it
follows that $a=z_1$. Therefore $x\dom a$, and Lemma~\ref{lem:basicdom}~(v) then
yields that  $(x, a)\geq 1$.

Conversely, suppose that $x\in D_n$ and $a\in \Pi$ such that $(x, a)\geq 1$.
Then Lemma~\ref{lem:basicdom}~(i) yields that $x\dom a$; furthermore,
Lemma~\ref{lem:pre} yields that $r_a x\prec x$. Hence Lemma~\ref{lem:lemma13}
yields that 
\begin{equation}
\label{eq:Dn}
r_a D(r_a x)\subseteq D(x).
\end{equation}
Now suppose for a contradiction that $r_a x\notin D_{n-1}$. Then
Lemma~\ref{lem:lemma11} yields that $r_a x \in D_n\uplus D_{n+1}$. From
(\ref{eq:Dn}) it is clear that $r_a x \notin D_{n+1}$. But if $r_a x\in D_n$,
then (\ref{eq:Dn}) yields that $r_a D(r_a x) =D(x)$. Observe that $a\in D(x)$
and $a\notin r_aD(r_ax)$, producing a contradiction as desired.

(ii):\qquad Replace $x$ by $r_a x$ in (i) above then we may obtain the desired
result.

(iii):\qquad Follows from (i), (ii) and Lemma \ref{lem:lemma11}. 
  
\end{proof}

\begin{definition}
 \label{def:st}
For each $x\in \Phi^+$, define 
\begin{displaymath}
\begin{split}
S(x)&= \{\, w\in W \mid \ell(w)=dp(x)-1 \text{ and } w^{-1} x \in \Pi \,\},\\ 
T(x)&=\{\, w \in W \mid \ell(w)=dp(x) \text{ and } w^{-1} x \in \Phi^- \,\}.
\end{split}
\end{displaymath}

\end{definition}
In other words, for $x\in \Phi^+$, $S(x)$ (respectively, $T(x)$) consists of all
$w\in W$ of minimal length with $w^{-1} x \in \Pi$ (respectively, $w^{-1}x\in
-\Pi$). Note that for each $w\in S(x)$, there exist some $w'\in T(x)$ and $a\in
\Pi$ such that $w'=w r_a$ with $\ell(w')=\ell(w)+1$.

\begin{proposition}
\label{pp:16}
Suppose that $x\in \Phi^+$ and let $w\in S(x)$ be arbitrarily chosen. Then $x\in
D_n$ where $n=\#\{\,b\in N(w^{-1})\mid (x,b)\geq 1\,\}$. In particular, the
integer $n$ is independent of the choice of $w \in S(x)$.
\end{proposition}

\begin{proof}
Let $x\in \Phi^+$ and write $x= w a$ where $ w\in S(x)$ and $a\in \Pi$. Let
$w=r_{a_1}\cdots r_{a_l}$ be such that $l=\ell(w)$ and $a_1, a_2, \cdots,
a_l\in\Pi$. Observe that for each $i\in \{\, 2, \ldots, l\,\}$, 
\begin{align}
\label{eq:l1}
w^{-1}(r_{a_1}r_{a_2}\cdots r_{a_{i-2}})a_{i-1}&=r_{a_l}\cdots r_{a_1}
r_{a_1}\cdots r_{a_{i-2}}a_{i-1}\nonumber\\
&=r_{a_l}\cdots r_{a_i} r_{a_{i-1}} a_{i-1}\nonumber\\
&=-r_{a_l}\cdots r_{a_i}a_{i-1}.
\end{align} 
Under our assumptions $\ell(r_{a_l} r_{a_{l-1}}\cdots r_{a_i}
r_{a_{i-1}})=\ell(r_{a_l} \cdots r_{a_i} )+1$ and $\ell (r_{a_1} r_{a_2}\cdots
r_{a_{i-2}}r_{a_{i-1}})=\ell(r_{a_1}r_{a_2}\cdots r_{a_{i-2}})+1$,  hence
Proposition~\ref{pp:anu3}~(i) yields that $r_{a_l}\cdots r_{a_i}a_{i-1}\in
\Phi^+$ and $r_{a_1}r_{a_2}\cdots r_{a_{i-2}}a_{i-1}\in \Phi^+$. Thus
(\ref{eq:l1}) yields that
\begin{equation}
\label{eq:add}
 (r_{a_1}r_{a_2}\cdots r_{a_{i-2}})a_{i-1}\in N(w^{-1}).
\end{equation}
Now by Proposition \ref{pp:14}, we can immediately deduce that $x\in D_n$ where 
\begin{displaymath}
    \begin{split}
      n &= \#  \{\, i \mid (a_{i-1}\ ,\ r_{a_i} r_{a_{i+1}} \cdots r_{a_l} a)
\leq -1      \,\}\\
      &= \# \{\,  i\mid (r_{a_1}\cdots r_{a_{i-1}}(a_{i-1})\ ,\ r_{a_1}\cdots
r_{a_l}(a)) \leq -1 \,\}\\
      &= \# \{\,   i\mid  (r_{a_1} \cdots r_{a_{i-1}}(a_{i-1})\ ,\ x) \leq -1   
              \,\} \\
      &= \# \{\, i\mid  (-r_{a_1} \cdots r_{a_{i-2}}(a_{i-1})\ ,\ x) \ \leq -1  
           \,\}\\
      &= \# \{\, b\in N(w^{-1})\mid (-b, x)\leq -1              \,\}\\
      &= \# \{\,  b\in N(w^{-1}) \mid (b, x)\geq 1         \,\}.
    \end{split}
  \end{displaymath}
Lemma \ref{lem:basicdom} (v) then yields that either $x \dom b$ or $b\dom x$.
Since all such $b$ are in $N(w^{-1})$ where $w\in S(x)$, it follows that $w^{-1}
x \in \Pi$ and $w^{-1}b \in \Phi^-$. Thus $b$ cannot dominate $x$. So we may
conclude that $x\in D_n$, where 
 \begin{equation}
    \label{eq:2}
    n = \#\{\, b\in N(w^{-1})\mid x \dom b   \,\},
  \end{equation}
for all $w\in S(x)$.
But \eqref{eq:2} says precisely that $D(x)\subseteq N(w^{-1})$ and 
\begin{displaymath}
\begin{split}
D(x)&= \{\, b \in N(w^{-1} )\mid x\dom b\, \} \\
    &=\{ \, b \in N(w^{-1}) \mid (x, b)\geq 1\, \}. 
\end{split}
\end{displaymath}
\end{proof}

From the above proof we immediately have:

\begin{corollary}
Let $x \in \Phi^+$. Then 
$D(x) \subseteq \bigcap\limits_{w \in S(x)}N(w^{-1})$.
\qed
\end{corollary}

It turns out that we can also say something about the roots in $\bigcap_{w\in
S(x)}N(w^{-1})\setminus D(x)$. Indeed in the next two lemmas we deduce that if
$b\in \bigcap_{w\in S(x)} N(w^{-1})$, then $(x, b)> 0$.

\begin{lemma}
\label{lem:T}
Suppose that $x \in \Phi^+$, $w\in T(x)$ and $b \in N(w^{-1})$. Then $(b, x) >
0$.
\end{lemma}

\begin{proof}
If $\dep(x)=1$ then $x\in \Pi$, whence $T(x)=\{\,r_x\,\}$ and $x=b$, and so
$(b,x)=1$ as required.
Thus we may assume that $\dep(x)>1$ and proceed by an induction on $\dep(x)$.
Let $a\in \Pi\cap N(w^{-1})$. Then 
$$
\ell(r_a w)=\ell(w^{-1}r_a)=\ell(w^{-1})-1 =\ell(w)-1.
$$
Now since $(r_aw)^{-1}(r_a x)=w^{-1}x \in \Phi^-$, it follows that 
$$
\dep(r_a x)\leq \ell(r_a w)<\ell(w)=\dep(x),
$$
and hence Lemma \ref{lem:pre} yields that $(a, x)>0$. If $b=a$ then we are done,
thus we may assume that $b\neq a$ (in particular, $r_a b\in \Phi^+$) and let
$w'=r_a w$. Observe that then $w'\in T(r_a x)$. Since $b\in N(w^{-1})$, it
follows that  $r_a b\in N(w'^{-1})$ and so the inductive hypothesis yields that
$(r_a b, r_a x)> 0$. Finally since $(\,,\,)$ is $W$-invariant, it follows that
$(b, x)>0$ as required. 
\end{proof}

\begin{lemma}
\label{lem:S}
Suppose that $x \in \Phi^+$, $w\in S(x)$ and $b \in N(w^{-1})$. Then $(b, x) >
0$.
\end{lemma}
\begin{proof}
Follows from Lemma \ref{lem:T} and the fact that for each $w \in S(x)$ there is
a $w' \in T(x)$ such that $N(w^{-1}) \subset N(w'^{-1})$.
\end{proof}

\begin{lemma}
\label{lem:18}
For $n \in \N$, if $D_n = \emptyset$, then $D_m =\emptyset$ for all $m\in \N$
such that $m>n$.
\end{lemma}

\begin{proof}
Suppose for a contradiction that there exists $n\in \N$ such that $D_n =
\emptyset$ and yet $D_{n+1} \neq \emptyset$. Let $x \in D_{n+1}$. Then
Lemma~\ref{lem:lemma11} yields that $r_a x \in D_{n+1}\uplus D_{n+2}$, for all
$a\in \Pi$. Furthermore, Lemma~\ref{lem:lemma13} yields that if $a\in \Pi$ such
that $r_a x \prec x$ then $r_a x \in D_{n+1}$ still. Write $x= wb$, where  $b\in
\Pi$, and $w\in S(x)$. Suppose that $w=r_{a_1}r_{a_2} \cdots r_{a_l}$ with
$\ell(w)=l$ and $a_1, a_2, \ldots, a_l\in \Pi$. Then $r_{a_i}\cdots r_{a_2}
r_{a_1} x \in D_{n+1}$, for all $i \in \{\, 1, \ldots , l \,\}$, and in
particular, $b = r_{a_l} \cdots r_{a_1} x \in D_{n+1}$, contradicting the fact
that $b \in \Pi\subset D_0$.
\end{proof}

\begin{corollary}
\label{cor:c19}
Let $W$ be an infinite Coxeter group with $\#R<\infty$. Then for each
nonnegative integer $n$, the corresponding $D_n$ is non-empty.
\end{corollary}

\begin{proof}
It is clear from the definition of the $D_n$'s that $\Phi^+ = \biguplus_{n \geq
0} D_n$. Since $W$ is an infinite Coxeter group, Proposition~\ref{pp:anu3}~(iii)
yields that $\#\Phi^+ = \infty$. On the other hand, since $\#R<\infty$,
Theorem~\ref{th:dnfinite} yields that for each nonnegative integer $n$, $\#D_n
<\infty$. Thus the desired result follows from Lemma \ref{lem:18}.
\end{proof}

The following is a generalization of Proposition \ref{pp:14}:

\begin{proposition}
\label{pp:general}
Suppose that $x\in D_n$ with $n > 0$, and let $a\in \Phi^+$. Then  
\begin{itemize}
\item [(i)] $\#D(r_a x) < n$  if $(x,a) \geq 1$;
\item [(ii)] $\#D(r_a x)  > n$  if $(x,a) \leq -1$.

\end{itemize}
\end{proposition}

\begin{proof}

(i)\qquad If $\dep(a)=1$ then this is just Proposition \ref{pp:14}. Hence we may
assume that $\dep(a) > 1$, and proceed by an induction on $\dep(a)$.

Write $a = r_b c$ where $b\in \Pi$ and $c\in \Phi^+$. Then $r_a =r_b r_c r_b$.
Furthermore, suppose that  
 \begin{equation}
    \label{eq:3}
    \dep(a) = \dep(c)+1.
  \end{equation}
Now since $(x, a)=(x, r_b c)=(r_b x, c) \geq 1$, it follows from the inductive
hypothesis that
\begin{equation}
\label{eq:4}
\# D(r_c(r_b x))   <\#D(r_b x).
\end{equation}
Then we have three possibilities to consider:
\begin{itemize}
\item [1)] $(b , x) \geq 1$; 
\item [2)] $(b , x) \leq -1$;
\item [3)] $(b , x) \in (-1,1 )$.
\end{itemize}

If 1) is the case, then Proposition \ref{pp:14} yields that $r_b x \in D_{n-1}$
and hence 
\begin{align*}
    \#D(r_a x) &= \# D(r_b(r_c r_b x))   \\
               &\leq \#D(r_c (r_b x))+1 \qquad&\text{ (follows from Lemma
\ref{lem:lemma11} )}\\
    &\leq \#D(r_b x)    &\text{ (follows from \eqref{eq:4})}        \\
    &= n-1,
\end{align*}
as required.

If 2) is the case, then Proposition \ref{pp:14} yields that $r_b x \in D_{n+1}$,
and $(b, r_c(r_b x))=(b, r_bx-2(r_b x, c\,)c\,)=(b, r_b x)-2(x, a)(b,c)$.
Observe that Lemma~\ref{lem:pre} and \eqref{eq:3} together yield that $(b,c)<0$
and since by assumption $(x, a)\geq 1$, it follows that
\begin{equation}
\label{eq:5}
 (b, r_c(r_b x))>(b, r_b x)\geq 1. 
\end{equation}
Then  
\begin{align*}
\#D(r_a x)& = \#D(r_b (r_c r_b x))\\
          & = \#D(r_c r_b x)-1\quad &\text{(by \eqref{eq:5} above and
Proposition \ref{pp:14})}\\
          & \leq \#D(r_b x)-2 &\text{ (by \eqref{eq:4})}\\
          & \leq n-1 &\text{(since $r_b x \in D_{n+1}$ in case 2))}
\end{align*} 
as required.

If 3) is the case, then we are done unless $\#D(r_c (r_b x))= n-1$ together with
$(b,r_c r_b x ) \leq -1$. But this is impossible, since 
\begin{displaymath}
\begin{split}
 (b, r_c r_b x)&=(b, r_b x)-2(r_bx, c)(b,c)\\
               &=-(b, x)-2\underbrace{(a,x)}_\text{$\geq
1$}\underbrace{(b,c)}_\text{$<0$}\\
               &> -1.
\end{split}
\end{displaymath}
Thus $\#D(r_a x) = \#D(r_b r_c r_b x) <n$ in this case too. This completes the
proof of $(i)$.

(ii) Replace $x$ by $r_a x$, then apply (i) above.

\end{proof}

\begin{lemma}
\label{lem:lemma21}
Suppose that $x\in D_n$ with $n\geq 1$. Then there exists some $y\in D_{n-1}$
with $y \prec x$.
\end{lemma}

\begin{proof}
Suppose that the contrary is true. Let $x\in D_n$ such that there is no root in
$D_{n-1}$ preceding $x$. Write $x=wa$, where $a\in \Pi$, and $w\in S(x)$. Let
$w=r_{a_1} r_{a_2}\cdots r_{a_l}$ for some $a_1, \cdots, a_l \in \Pi$ with
$\ell(w)=l$. Then $a = r_{a_l}\cdots r_{a_1}x$. Observe that then
\begin{equation} 
\label{eq:preseq}
a\prec r_{a_{l-1}}\cdots r_{a_{1}}x\prec r_{a_{l-2}}\cdots
r_{a_1}x\prec\cdots\prec r_{a_1}x\prec x.
\end{equation} 
The assumption that $x$ is not preceded by any root in $D_{n-1}$, together with
Proposition~\ref{pp:14} yield that 
all the roots in (\ref{eq:preseq}), including $a$, are in $D_n$, contradicting
the fact the $a\in \Pi\subseteq D_0$.

\end{proof}

Next we give an algorithm to systematically compute all the $D_n$'s for an
arbitrary Coxeter group $W$ of finite rank:

\begin{proposition}
\label{pp:22}
Suppose that $W$ is a Coxeter group of finite rank. For $n\geq 1$, there is an
algorithm to compute $D_n$ provided that $D_{n-1}$ is known.
\end{proposition}
\begin{proof}
We outline such an algorithm:
\begin{itemize}
\item[1)] Set $D= \emptyset$. 
\item[2)] Enumerate all the elements of $D_{n-1}$ in some order, that is, write
$D_{n-1} =\{ x_1, \cdots, x_m\}$, where $m=\#D_{n-1}$.
\item [3)] Starting with $x_1$, apply all the reflections $r_a$ where $a \in
\Pi$, to $x_1$, one at a time. If $(a, x_1) \leq -1$, then add $r_a x_1$ to $D$
if it is not already in $D$.
\item[4)] Repeat $3)$ to $x_2, \cdots , x_m$.
\item[5)] Enumerate all the elements of the modified set $D$ in some order, that
is, write $D=\{ x'_{1}, x'_{2}, \cdots, x'_{\#D}\}$.
\item[6)] Starting with $x'_{1}$, apply all the reflections $r_a$ where $a\in
\Pi$, to $x'_{1}$, one at a time. If $(a, x'_{1}) \in (-1, 0)$ and $r_a x'_{1}
\notin D$, then add $r_a x'_{1}$ to $D$.
\item[7)] Repeat $6)$ to $x'_{2}, \cdots, x'_{\#D}$.
\item[8)] Repeat steps $5)$ to $7)$ above.
\item[9)] Repeat $8)$ until no new elements can be added to $D$.
\item[10)] Set $D_n = D$.
\end{itemize} 
 
Next we show that the above algorithm will be able to produce all elements of
$D_n$ within a finite number of iterations.

Let $x\in D_n$ ($n\geq 1$) be arbitrary. Lemma \ref{lem:lemma21} yields that
there exists a $y\in D_{n-1}$ with $y\prec x$. Write $x=wy$ for some $w\in W$
with $\ell(w)=\dep(x)-\dep(y)$. Let $w=r_{a_1}r_{a_2}\cdots r_{a_l}$ where $a_1,
\ldots, a_l\in \Pi$ and $\ell(w)=l$. Then
$$
y\prec r_{a_l}y \prec r_{a_{l-1}} r_{a_l} y \prec \cdots \prec
r_{a_1}r_{a_2}\cdots r_{a_l} y=x.
$$
Since $x\in D_n$ and $y\in D_{n-1}$, it follows from Lemma~\ref{lem:lemma13}
that 
$$
r_{a_l}y,\, r_{a_{l-1}}r_{a_l}y,\, \ldots,\, r_{a_2}r_{a_3}\cdots r_{a_l} y \in
D_{n-1}\uplus D_{n}.
$$
Therefore there exists $i\in \{\, 1, 2, \ldots, l\,\}$ such that 
\begin{align*}
y&\in D_{n-1}\\
r_{a_l}y&\in D_{n-1}\\
&\vdots\\
r_{a_{i+1}} r_{a_{i+2}}\cdots r_{a_l} y&\in D_{n-1}\\
\noalign{\hbox{and}}
r_{a_i} (r_{a_{i+1}} r_{a_{i+2}}\cdots r_{a_l} y)&\in D_n\\
r_{a_{i-1}}r_{a_i} (r_{a_{i+1}} r_{a_{i+2}}\cdots r_{a_l} y)&\in D_n\\
&\vdots\\
r_{a_1}r_{a_2}\cdots r_{a_l} y=x&\in D_n.
\end{align*}
Since $r_{a_{i+1}} r_{a_{i+2}}\cdots r_{a_l} y\in D_{n-1}$, it follows that
$r_{a_i}r_{a_{i+1}} r_{a_{i+2}}\cdots r_{a_l} y$ is an element of $D_n$
obtainable by going through steps  3) and 4) above. This in turn implies that
$r_{a_{i-1}} r_{a_i}\cdots r_{a_l} y$ is an element obtainable by going through
steps 5) to~7). It then follows that $r_{a_{i-2}}r_{a_{i-1}} r_{a_i}\cdots
r_{a_l} y$ and so on are all obtainable by (repeated) application of step~8). In
particular, $x=r_{a_1}\cdots r_{a_l}y$ can be obtained after $(i-2)$ iterations
of step~8). Thus $x$ can be obtained by going through steps~1) to~8), with
step~8) repeated finitely many times. Since $x\in D_n$ was arbitrary, it follows
that every element of $D_n$ can be obtained from the above algorithm in this
manner with step~8) repeated finitely many times. 

Finally, $W$ is of finite rank, so $\#D_n<\infty$ and $\#D_{n-1}<\infty$.
Therefore step~9) will only be repeated a finite number of times and hence the
algorithm will terminate completing the proof.
\end{proof}

\begin{corollary}
If $\#R< \infty$, then we may compute $D_n$, for all $n\in \N$.
\end{corollary}
\begin{proof}
\cite{BB98} gives a complete description of $D_0$ when $\#R<\infty$. Now combine
\cite{BB98} and Proposition \ref{pp:22}, the result follows immediately.
\end{proof}

\section{Acknowledgments}
The results presented in this paper are based on parts of the
author's PhD thesis~\cite{FU1} and the author wishes to thank
A/Prof.~R.~B.~Howlett for all his help and encouragement. The author
also wishes to thank Prof.~G.~I.~Lehrer and Prof.~R.~Zhang for supporting this work.
Due gratitude must also be paid to the referee of this paper for
a number of valuable suggestions. 

\bibliographystyle{amsplain}

\providecommand{\bysame}{\leavevmode\hbox to3em{\hrulefill}\thinspace}
\providecommand{\MR}{\relax\ifhmode\unskip\space\fi MR }
\providecommand{\MRhref}[2]{%
  \href{http://www.ams.org/mathscinet-getitem?mr=#1}{#2}
}
\providecommand{\href}[2]{#2}

\end{document}